\documentclass[12pt]{amsart}

\usepackage{amssymb}

\newtheorem{theorem}{\bf  Theorem}
\newtheorem{lemma}{\bf  Lemma}
\newtheorem{statement}{\bf \sc Statement}
\newtheorem{proposition}{\bf  Proposition}
\newtheorem{corollary}{\bf \sc Corollary}
\newtheorem{remark}{ \sc Remark}



\begin{document}


\begin{center}
    {\bf A New Method of Matrix Spectral Factorization  \footnote{This paper
    includes the detailed proofs for an
innovative method for matrix spectral factorization that can be
used in numerous applications, including Filtering, Data
Compression, and Wireless Communications. A U.S. patent
application has been submitted for this innovation through the
Technology
Commercialization Center of the University of Maryland.}}\\[5MM]

 G.~Janashia$^\dag$, E.~Lagvilava, and  L.~Ephremidze
        \end{center}

\vskip+0.8cm {\bf  {\em Abstract}---A new  method of matrix
spectral factorization is proposed which reliably computes an
approximate spectral factor of any matrix spectral density that
admits spectral factorization. }

\vskip+0.2cm {\bf {\em Index Terms}---Matrix spectral
factorization algorithm.}

\vskip+1cm

\section{Introduction}

Spectral factorization plays a prominent role in a wide range of
fields in Communications, System Theory, Control Engineering and
so on. In the scalar case  arising for single input and single
output systems, the factorization problem is relatively easy and
several classical algorithms exist to tackle it (see the survey
paper [17]) together with  reliable information on their software
implementations [8]. There are also  some recent claims as to
their improvement [2]. Matrix spectral factorization which arises
for multi-dimensional systems is essentially more difficult (see
Sect. 2, where the mathematical reasons of this fact are
explained). Since Wiener's original efforts [19] to create a sound
computational method of such factorization, tens of different
algorithms have appeared in the literature (see the survey papers
[16], [17] and the references therein), but none of them is
thought to have an essential superiority over all others (see [16,
p. 1077], [14, p. 206]). Besides, most of these algorithms impose
extra restrictions on matrix spectral densities (e.g., to be real
or rational or nonsingular on the boundary), while the
Paley-Wiener necessary and sufficient condition (see (2)) will do
for the existence of spectral factorization (see Sect. 2).

In the present paper,  a  new computational method of matrix
spectral factorization is developed. The proposed algorithm can be
applied to any matrix spectral density  satisfying the
Paley-Wiener condition. It should be said that the branch of
mathematics where the spectral factorization problem is posed in
its general non-rational setting (see Sect. 2) is  the theory of
Hardy spaces (see Sect. 3),  and this method is completely worked
out in the framework of Hardy spaces, which added to its
effectiveness.

To describe our method of $r\times r$ matrix spectral
factorization in a few words,  it carries out spectral
factorization of $m\times m$ left-upper submatrices step-by-step,
$m=1,2,\ldots,r$. It is shown that in this process the decisive
role is played by  unitary matrix functions of certain structure
(see Theorem 1), which removes many technical difficulties
connected with computation. The explicit construction of such
matrices in Theorem 2 is an essential component of the algorithm.
A close relationship of these unitary matrix functions with
compactly supported wavelets has recently been discovered, which
makes it possible to construct compact wavelets in a fast and
reliable way and to completely parameterize them (see [6]).

Preliminary numerical simulations confirm the potential of the
proposed algorithm (see Sect. 7).

The algorithm was announced in [3] and,  for second order
matrices, described in [12].

\section{Formulation of the problem}

A series of papers [18], [19], [10], [11] led to the following

\smallskip

{\bf Wiener Matrix Spectral Factorization Theorem:} Let
\begin{equation}
S(t)=\begin{pmatrix} s_{11}(t)& s_{12}(t)& \cdots&s_{1r}(t)\\
s_{21}(t)& s_{22}(t)& \cdots&s_{2r}(t)\\
\vdots&\vdots&\vdots&\vdots\\s_{r1}(t)& s_{r2}(t)&
\cdots&s_{rr}(t)\end{pmatrix},
\end{equation}
$t\in {\mathbb T}$, be a  positive definite (a.e.) integrable
matrix function,  $0<S(t)\in L_1({\mathbb T})$, which satisfies
the  condition
\begin{equation}
\log \det S(t)\in L_1({\mathbb T}).
\end{equation}
Then it admits a {\em spectral factorization}
\begin{equation}
    S(t)=S^+(t)S^-(t)=S^+(t)\big(S^+(t)\big)^*,
\end{equation}
where $S^+$ is an $r\times r$ outer analytic  matrix function from
the Hardy space $H_2$ and
$S^-(z)=\big(S^+(1/\overline{z})\big)^*$, $|z|>1$. It is assumed
that (3) holds a.e. on $\mathbb{T}$. (The factorization (3) is
called {\em left} since the analytic inside $\mathbb{T}$ factor
stands on the left-hand side. The {\em right} spectral
factorization of $S$ can be obtained by the left factorization of
$S^T$.)

\smallskip

The sufficient condition (2) is also a necessary one for the
factorization (3) to exist (see Sect. 3).

A spectral factor $S^+(z)$ is unique up to a constant right
unitary multiplier (see, e.g., [5]), and the unique spectral
factor with an additional requirement that $ S^+(0)$ be positive
definite is called {\em canonical}.

After the proof of the existence of matrix spectral factorization,
the computation of the spectral factor for a given matrix spectral
density has become a challenging problem due to its applications
in practice.

In the scalar case, $r=1$, the canonical spectral factor $S^+\in
H_2$ can be explicitly written by the formula (see, e.g., [20;
VII, 7.33])
\begin{equation}
 S^+(z)=\exp\left(\frac 1{4\pi}
\int\nolimits_0^{2\pi}\frac{e^{i\theta}+z}{e^{i\theta}-z}\log
S(e^{i\theta})\,d\theta\right)
\end{equation}
and it is relatively easy to compute $S^+$ approximately. However,
there is no analog of this formula  in the matrix case because,
generally speaking, $e^{A+B}\not=e^Ae^B$ for non-commutative
matrices $A$ and $B$. This is the main reason for which the
approximate computation of the spectral factor $S^+$ in (3) for
the matrix spectral density (1) is essentially more difficult. The
present paper provides an algorithm for such computation.

The proposed method does not contribute to the improvement of
(numerical) scalar spectral factorization, but employs it to
fulfill matrix spectral factorization.

\section{Notation and Conventions}

Let $\mathbb{D}=\{z\in\mathbb{C}:|z|<1\}$, and
$\mathbb{T}=\partial\mathbb{D}$ be the unit circle. As usual,
$L_p=L_p(\mathbb{T})$, $0<p\leq\infty$, denotes the Lebesgue space
of $p$-integrable complex functions defined on $\mathbb{T}$.
$H_p=H_p(\mathbb{D})$, $0<p\leq\infty$, is the Hardy space of
analytic functions in $\mathbb{D}$ ,
$$
H_p=\left\{f\in\mathcal{A}(\mathbb{D}):\sup\limits_{r<1}
\int\nolimits_0^{2\pi}|f(re^{i\theta})|^p\,d\theta<\infty\right\}
$$
($H_\infty$ is the space of bounded analytic functions), and
$L_p^+=L_p^+(\mathbb{T})$ denotes the class of their boundary
functions. (All the relations for functions from
$L_p^+(\mathbb{T})$ or $L_p(\mathbb{T})$ are assumed to hold
almost everywhere.) Since there is a one-to-one correspondence
between $H_p(\mathbb{D})$ and $L_p^+(\mathbb{T})$, $p>0$ (see,
e.g., [20; VII, 7.25]), we naturally regard these two classes as
identical, and thus we can speak about the values of $f\in
L_p^+(\mathbb{T})$ inside the unit circle. Furthermore, we always
use the argument $t$ for the functions defined on $\mathbb{T}$ and
the argument $z$ for the functions defined in $\mathbb{D}$, so
that the boundary function of $f=f(z)\in H_p(\mathbb{D})$ is
denoted by $f=f(t)\in L_p^+(\mathbb{T})$ and we write
$f(z)|_{z=t}=f(t)$ when we wish to point out this fact. If we
write only $f$, its domain will be clear from the context.

We have $\log |f(t)|\in L_1(\mathbb{T})$ for each $0\not\equiv
f\in H_p$, $p>0$ \big(see, e.g., [20; VII, 7.25]\big), which
readily implies the necessity of the condition (2) for the
factorization (3) to exist since $L_1(\mathbb{T})\ni \log|\det
S^+(t)|=\frac12\log \det S(t)$.

The $n$th Fourier coefficient of an integrable function $f\in
L_1(\mathbb{T})$ is denoted by $c_n(f)$. For $p\geq 1$,
$L_p^+(\mathbb{T})$ coincides with the class of functions from
$L_p(\mathbb{T})$ whose  Fourier coefficients with negative
indices are equal to zero. We also deal  with
$L_p^-(\mathbb{T})=\{\overline{f}:f\in L_p^+(\mathbb{T})\}=\{f\in
L_p(\mathbb{T}):c_n(f)=0 \text{ whenever }n>0\}$.  The set of
trigonometric polynomials is denoted by $\mathcal{P}$, i.e. $f\in
\mathcal{P}$ if $f$ has only a finite number of nonzero Fourier
coefficients. Also let $\mathcal{P}^\pm:=\mathcal{P}\cap
L_\infty^\pm$, $\mathcal{P}_N:=\{f\in \mathcal{P}:c_n(f)=0 \text{
whenever }|n|>N\}$, and $\mathcal{P}_N^\pm=\mathcal{P}_N\cap
\mathcal{P}^\pm$. Obviously, $f\in \mathcal{P}_N^+ \Leftrightarrow
\overline{f}\in \mathcal{P}_N^-$.

For $f(t)=\sum_{n=-\infty}^\infty c_nt^n\in L_2(\mathbb{T})$, let
$P^+f(t)$, $P^-f(t)$, and $P_Nf(t)$ be the projections
$\sum_{n=0}^\infty c_nt^n$, $\sum_{n=-\infty}^0 c_nt^n$, and
$\sum_{n=-N}^N c_nt^n$, respectively, on $L_2^+(\mathbb{T})$,
$L_2^-(\mathbb{T})$, and $\mathcal{P}_N$.

The superscript  "+" (resp. "$-$") of a function $f^+$ (resp.
$f^-$) emphasizes that this function belongs to $L_p^+$ (resp.
$L_p^-$).

The norms $\|\cdot\|_{L_p}$ and $\|\cdot\|_{H_p}$ are defined in a
usual way.

If $M $ is a matrix, then $\overline{M}$ denotes the matrix with
conjugate entries and $M^*:=\overline{M}^T$. If $M$ is positive
definite, $M>0$, then the unique $M_0>0$ that satisfies
$M_0M_0^*=M$ is denoted by $\sqrt{M}$.

If $M$ is an $r\times r$ matrix and $m\leq r$, then $(M)_{m\times
m}$ is assumed be the $m\times m$ upper-left  submatrix of $M$.

An $r\times r$ matrix $U$ is called unitary if $UU^*=U^*U=I_r$,
where $I_r$ stands for the $r$-dimensional unit matrix. Obviously
the entries of a unitary matrix are bounded by 1.

A matrix function $M(t)$ defined on $\mathbb{T}$ is called
positive definite or unitary if it is such  for almost all
$t\in\mathbb{T}$. $M(t)$ is said to belong to some class, say,
 $L_p^+(\mathbb{T})$ (we write $M(t)\in L_p^+(\mathbb{T})$) if its
entries belong to this class. $P_NM(t)$ denotes the matrix
function whose entries are the projections of the entries of
$M(t)$ on $\mathcal{P}_N$. A sequence of matrix functions is said
to be convergent in  $L_p$-norm if their entries are convergent in
this norm.

The class of $r\times r$ unitary   matrix functions $U(t)$,
\begin{equation}
U(t)U^*(t)=I_r \;\;\text{a.e.},
 \end{equation}
is denoted by $\mathcal{U}_r(\mathbb{T})$, and
$\mathcal{S}\mathcal{U}_r(\mathbb{T})$ stands for the subclass of
those $U(t)\in \mathcal{U}_r(\mathbb{T})$  the determinants of
which are equal to 1,
\begin{equation}
\det U(t)=1 \;\;\text{a.e.}
 \end{equation}

The set of outer analytic functions  from the Hardy space $H_p$,
$p>0$, is denoted by $\mathcal{O}_p$. Recall that $f\in
\mathcal{O}_p$ if and only if $0\not\equiv f\in H_p$ and
$$
f(z)=c\cdot  \exp\left(\frac 1{2\pi}
\int\nolimits_0^{2\pi}\frac{e^{i\theta}+z}{e^{i\theta}-z}\log
\big|f(e^{i\theta})\big|\,d\theta\right),\;\;\;\;\;|c|=1.
$$
(From this definition and  H\"{o}lder's inequality it follows that
if $f\in \mathcal{O}_p$ and $g\in \mathcal{O}_q$, then $fg\in
\mathcal{O}_{(p+q)/pq}$\,.) Clearly, $f(z)\not=0$ for each $z\in
\mathbb{D}$ and $|f(t)|>0$ for a.a. $t\in \mathbb{T}$ if $f\in
\mathcal{O}_p$. The set of functions $f\in \mathcal{O}_p$ which
are positive at the origin (which happens when $c=1$ in the above
definition) is denoted by $\mathcal{O}_p^0$. We say that a
$r\times r$ matrix function $M(t)\in H_p$, $p\geq 0$, is outer if
its determinant belongs to
 $\mathcal{O}_{p/r}$. This definition coincides with some other equivalent
 definitions of outer matrix functions (see, e.g., [11]). $M(t)\in
 \mathcal{O}_{p/r}^0$ means that $M(0)>0$ in addition.

 $f_n\rightrightarrows f$ means that $f_n$ converges to $f$ in measure.

$\langle\cdot,\cdot\rangle_m$ and $\|\cdot\|_{\mathbb{C}^m}$
denote the usual scalar product and the norm, respectively, in the
$m$-dimensional complex space $\mathbb{C}^m$.

$\delta_{ij}$ stands for the  Kronecker delta.

\smallskip

To conclude the section, we formulate a simple statement from the
Lebesgue integral theory in the best suitable form  for further
references.
\begin{statement}
Let $f_n(t)\in L_2(\mathbb{T})$, $n=1,2,\ldots$,
$\|f_n(t)-f(t)\|_{ L_2}\to 0$, $u_n(t)\in L_\infty(\mathbb{T})$,
$u_n(t)\leq 1$, $n=1,2,\ldots$, and $u_n(t)\rightrightarrows
u(t)$. Then $\|f_n(t)u_n(t)-f(t)u(t)\|_{ L_2}\to 0$ $(see, e.g.,
[9;\S26, Th.\, 3])$.
\end{statement}

\section{Mathematical Background of the Method}

In this section we formulate some statements needed to describe
our method. Most of the proofs are given in the next sections.

The uniqueness of spectral factorization (3) mentioned in Sect. 2
means that $S^+(z)\cdot U$ is also a spectral factor for any
(constant) unitary matrix $U$, and if $S_1^+(z)$ and $S_2^+(z)$
are two spectral factors, then $S_1^+(z)= S_2^+(z) U$ for some
unitary matrix $U$ (see, e.g., [5]). Since for any $r\times r$
non-singular matrix $S$ there exists a unique unitary matrix $U$
which makes the product $SU$ positive definite (see, e.g. [7; IX
\S 14]),  the canonical spectral factor $S_c^+(z)$ (with an
additional requirement that $ S_c^+(0)$ be positive definite) is
unique. Namely,
\begin{equation}
S_c^+(z)=S^+(z)\big(S^+(0)\big)^{-1}\sqrt{S^+(0)(S^+(0))^*}
 \end{equation}
 for any spectral factor $S^+(z)$. (Other uniqueness restrictions
 on $S^+$ can be imposed so that $S^+(0)$ would be, for example,  lower triangular with positive
 entries on the diagonal.) The following lemma can be
 applied for the approximation of the canonical spectral factor
 after the approximate computation of an arbitrary  spectral
 factor.

\begin{lemma}
Let $S^+(t)$ be a spectral factor of $(1)$ and let $S_n^+(t)\in
H_2$, $n=1,2,\ldots$, be such that
 \begin{equation}
\|S_n^+(t)-S^+(t)\|_{H_2}\to 0\;\;\text{as}\;\;n\to \infty.
 \end{equation}
Then
\begin{equation}
\big\|S_n^+(z)\big(S_n^+(0)\big)^{-1}\sqrt{S_n^+(0)(S_n^+(0))^*}-S_c^+(z)\big\|_{H_2}\to
0\;\;\text{as}\;\;n\to \infty.
\end{equation}
\end{lemma}
\begin{proof}
Since $S^+(0)$ is non-singular and (8) implies that $S_n^+(0)\to
S^+(0)$, we have
$\sqrt{S_n^+(0)\big(S_n^+(0)\big)^*}\to\sqrt{S^+(0)\big(S^+(0)\big)^*}$
and  $\big(S_n^+(0)\big)^{-1}\to \big(S^+(0)\big)^{-1}$. Therefore
(9) follows from (8) and (7).
\end{proof}

The following lemma  is used several times throughout the paper.

\begin{lemma}
Let $M(t)$ be any $m\times m$ matrix function from
$L_2(\mathbb{T})$ satisfying
 \begin{equation}
\det M(t)\in \mathcal{O}_{2/m}\subset
H_{2/m}=L_{2/m}^+(\mathbb{T}).
 \end{equation}
 If
${U}(t)\in \mathcal{S}\mathcal{U}_m(\mathbb{T})$  is such that
\begin{equation}
 M(t){U}(t)\in L_2^+({\mathbb T})
 \end{equation}
holds, then $M(t){U}(t)$ is a spectral factor of $M(t)M^*(t)$.
\end{lemma}

\begin{proof}
Taking into account (5), we have
$$M(t){U}(t)\cdot\big(M(t){U}(t)\big)^*=M(t){U}(t){U}^*(t)M^*(t)=M(t)M^*(t).$$
In view of (11), $MU$ can be extended inside $\mathbb{T}$. Hence,
by virtue of (6),
$$
\det( M{U})(z)|_{z=t}=\det\big( M(t){U}(t)\big)=\det M(t)=\det
M(z)|_{z=t}
$$
Thus $\det (M{U})(z)$ is an outer analytic function (see (10)) and
lemma holds.
\end{proof}

This proof gives rise to

\begin{corollary}
Let $M(t)$ be any $m\times m$ matrix function from
$L_2(\mathbb{T})$ satisfying $\det M(t)\in L^+_{2/m}$. If
${U}(t)\in \mathcal{S}\mathcal{U}_m(\mathbb{T})$ is such that
$(11)$ holds, then
\begin{equation}
\det\big(M{U}\big)(z)=(\det M)(z),\;\;\;|z|<1.
\end{equation}
\end{corollary}
It should be  pointed out that on the left-hand side of (12) we
first extend $M{U}$  inside $\mathbb{T}$ and then compute its
determinant, while on the right-hand side we first take the
determinant of $M(t)$ and then extend it inside $\mathbb{T}$.

\smallskip

The following two theorems play a decisive  role in our method.

\begin{theorem}
For every $m\times m$ matrix function $F(t)\in L_2(\mathbb{T})$ of
the  form
\begin{equation}
F(t)=\begin{pmatrix}1&0&0&\cdots&0&0\\
          0&1&0&\cdots&0&0\\
           0&0&1&\cdots&0&0\\
           \vdots&\vdots&\vdots&\vdots&\vdots&\vdots\\
           0&0&0&\cdots&1&0\\
           \zeta_{1}(t)&\zeta_{2}(t)&\zeta_{3}(t)&\cdots&\zeta_{m-1}(t)&f^+(t)
           \end{pmatrix},
\end{equation}
where
\begin{equation}
\zeta_j(t)\in L_2({\mathbb T}),\;j=1,2,\ldots, m-1,\text{ and }
f^{+}(t)\in\mathcal{O}_2^0\subset  L^{+}_2({\mathbb T}),
\end{equation}
there exists a unique $U_F(t)\in \mathcal{S}
\mathcal{U}_m(\mathbb{T}) $ of the form
\begin{gather}
U_F(t)=\begin{pmatrix}u^+_{11}(t)&u^+_{12}(t)&\cdots&u^+_{1,m-1}(t)&u^+_{1m}(t)\\
                 u^+_{21}(t)&u^+_{22}(t)&\cdots&u^+_{2,m-1}(t)&u^+_{2m}(t)\\
           \vdots&\vdots&\vdots&\vdots&\vdots\\
           u^+_{m-1,1}(t)&u^+_{m-1,2}(t)&\cdots&u^+_{m-1,m-1}(t)&u^+_{m-1,m}(t)\\[3mm]
           \overline{u^+_{m1}(t)}&\overline{u^+_{m2}(t)}&\cdots&\overline{u^+_{m,m-1}(t)}&\overline{u^+_{mm}(t)}\\
           \end{pmatrix},\\
u^+_{ij}(t)\in L_\infty^{+}({\mathbb T}),\;\;i,j=1,2,\ldots,m,
\end{gather}
 such that
  \begin{equation}
F(t) U_F(t)=F_c^+(t)\in\mathcal{O}_2^0\subset L^{+}_2({\mathbb
T}),
\end{equation}
where $F_c^+(t)$ is the canonical spectral factor of $F(t)F^*(t)$.
\end{theorem}

The proof of Theorem 1 relying on the existence of  spectral
factorization is relatively easy (see [3]). The core of the
proposed matrix spectral factorization method is the constructive
proof of Theorem 1 based on the following idea. We approximate
$F(t)$ in $L_2$ cutting off the tails of Fourier expansions of the
functions $\zeta_j(t)$, $j=1,2,\ldots,m-1$, and $f^+(t)$. Namely,
for a matrix function of the form (13), (14), let $F^{(N)}(t)$ be
$P_NF(t)$, i.e.
\begin{equation}
F^{(N)}(t)=\begin{pmatrix}1&0&0&\cdots&0&0\\
          0&1&0&\cdots&0&0\\
           0&0&1&\cdots&0&0\\
           \cdot&\cdot&\cdot&\cdots&\cdot&\cdot\\
           0&0&0&\cdots&1&0\\
           \zeta_{1}^{(N)}(t)&\zeta_{2}^{(N)}(t)&\zeta_{3}^{(N)}(t)&\cdots&\zeta_{m-1}^{(N)}(t)&f_{(N)}^+(t)
           \end{pmatrix},
\end{equation}
where
$$
f_{(N)}^+(t)=\sum_{n=0}^N c_n(f^+)t^n,\;\text{ and
}\;\zeta_j^{(N)}(t)=\sum_{n=-N}^N c_n(\zeta_j)t^n,
$$
 $j=1,2,\ldots,m-1$. It is obvious that,
 \begin{equation}
\|\zeta_j^{(N)}(t)-\zeta_j(t)\|_{L_2}\to
0,\;\;\|f_{(N)}^+(t)-f^+(t)\|_{L_2}\to 0
\end{equation}
or, equivalently,
\begin{equation}
\|F^{(N)}(t)-F(t)\|_{L_2}\to 0\;\;\text{as}\;\;N\to \infty.
\end{equation}
We will multiply (18) by the polynomial unitary matrix function
which eliminates the  Fourier coefficients with negative indices
of the product. Furthermore, we prove the following theorem for
matrix-functions $F^{(N)}(t)$, $N=1,2,\ldots$, which involves the
limiting case too.

\begin{theorem} $\mathbf{(a)}$
Let $N$ be any positive integer, and let a  matrix function
$F^{(N)}(t)\in \mathcal{P}_N$ of the form $ (18) $ be such that
\begin{equation}
\zeta_j^{(N)}(t)\in \mathcal{P}_N,\;j=1,2,\ldots, m-1,\text{ and }
f_{(N)}^{+}(t)\in \mathcal{P}_N^+,\;f_{(N)}^+(0)>0.
\end{equation}
 Then there exists and one can explicitly construct
\begin{equation}
U_{F^{(N)}}(t)\in \mathcal{S} \mathcal{U}_m(\mathbb{T})
\end{equation}
 of the form $ (15) $  such that
\begin{gather}
u^+_{ij}(t)\in \mathcal{P}_N^+ ,\;\;i,j=1,2,\ldots,m,
\end{gather}
  \begin{equation}
F^{(N)}(t) U_{F^{(N)}}(t)\in \mathcal{P}^+,
\end{equation}
and
 \begin{equation}
F^{(N)}U_{F^{(N)}}(0)>0.
\end{equation}

$\mathbf{(b)}$ Given an arbitrary sequence of  matrix functions
$F^{(N)}(t)$, $N=1,2,\ldots$, of the form $(18)$, $(21)$ which
converges in $L_2$ to $F(t)$ $($i.e. $ (20) $ holds$)$ of the form
$ (13) $, $ (14) $, we have
 \begin{equation}
\|F^{(N)}(t)U_{F^{(N)}}(t)-F_c^+(t)\|_{H_2}\to 0\;\;\text{ as }
N\to\infty.
\end{equation}
Furthermore, the sequence $U_{F^{(N)}}(t)$, $N=1,2,\ldots$,
 is convergent in measure. The
limiting matrix function $U(t)\in \mathcal{S}
\mathcal{U}_m(\mathbb{T})$ satisfies the conditions imposed on
$U_F(t)$ in Theorem 1, and therefore $U(t)$ coincides with
$U_F(t)$. Consequently, we have
\begin{equation}
U_{F^{(N)}}(t)\rightrightarrows U_F(t).
\end{equation}
\end{theorem}

The constructive proof of Theorem 2 $\mathbf{(a)}$ given in Sect
5, which computes explicitly and in a fast reliable way  the
coefficients of the functions $u^+_{ij}(t)\in \mathcal{P}_N^+$ in
 (23), is the essence of the proposed algorithm. The
 part $\mathbf{(b)}$ of the theorem involves the algorithm convergence
properties and is proved in Appendix A. We point out the fact that
Theorem 2 $\mathbf{(b)}$ includes also the proof of Theorem 1.

\section{A Constructive Proof of Theorem 2  $\mathbf{(a)}$ }

Throughout this section it is assumed that $N$ is fixed and
$\zeta_j(t):=\zeta_j^{(N)}(t)$, $f^+(t):=f_{(N)}^+(t)$, and
$F(t):=F^{(N)}(t)$.

For given functions $\zeta_j(t)$, $j=1,2,\ldots,m-1$, and $f^+(t)$
satisfying (21), we consider the following system of $m$
conditions, which plays a key role in the proof,
\begin{equation}
\begin{cases} \zeta_1(t)x^+_m(t)-f^+(t)\overline{x^+_1(t)}\in \mathcal{P}^+,\\
              \zeta_2(t)x^+_m(t)-f^+(t)\overline{x^+_2(t)}\in \mathcal{P}^+,\\
              \cdot\hskip+1cm \cdot\hskip+1cm \cdot\\
              \zeta_{m-1}(t)x^+_m(t)-f^+(t)\overline{x^+_{m-1}(t)}\in \mathcal{P}^+,\\
              \zeta_1(t)x^+_1(t)+\zeta_2(t)x^+_2(t)+\ldots+\zeta_{m-1}(t)x^+_{m-1}(t)
              +f^+(t)\overline{x^+_m(t)}\in \mathcal{P}^+,
         \end{cases}
\end{equation}
where the vector function
$\big(x^+_1(t),x^+_2(t),\ldots,x^+_m(t)\big)^T$ is unknown.

We say that a vector function
\begin{equation}
\mathbf{u}^+(t)=\big(u^+_1(t),u^+_2(t),\ldots,u^+_m(t)\big)^T\in
\mathcal{P}_N^+
\end{equation}
is a solution of (28) if and only if all the conditions in (28)
are satisfied whenever $x^+_i(t)=u^+_i(t)$, $i=1,2,\ldots,m$.
Observe that the set of solutions of (28) is a linear subspace of
$m$-dimensional vector-valued functions defined on $\mathbb{T}$.

For the vector function (29), we define the modified vector
function $\widetilde{\mathbf{u}^+}(t)$ as
\begin{equation}
\widetilde{\mathbf{u}^+}(t)=\big(u^+_1(t),u^+_2(t),\ldots,\overline{u^+_m(t)}\big)^T.
\end{equation}
It is assumed that the modification of (30) is (29).

We make essential use of the following
\begin{lemma}
Let $(21)$  hold and  let
\begin{equation}
\mathbf{u}^+(t)=\big(u^+_1(t),u^+_2(t),\ldots,u^+_m(t)\big)^T\in
\mathcal{P}_N^+ \text{ and }
\mathbf{v}^+(t)=\big(v^+_1(t),v^+_2(t),\ldots,v^+_m(t)\big)^T\in
\mathcal{P}_N^+
\end{equation}
be two $($possibly identical$)$ solutions of the system $(28)$.
Then
$\langle\widetilde{\mathbf{u}^+}(t),\widetilde{\mathbf{v}^+}(t)\rangle_m$
is the same for each $t\in \mathbb{T}$, i.e.
\begin{equation}
\sum_{i=1}^{m-1}u^+_i(t)\overline{v_i^+(t)}+\overline{u_m^+(t)}v_m^+(t)=\operatorname{const}.
\end{equation}
\end{lemma}
\begin{proof}
Substituting the functions $v^+$ into the first $m-1$ conditions
and the functions $u^+$ in the last condition of (28), and then
multiplying the first $m-1$ conditions by $u^+$ and the last
condition by $v_m^+$,  we get
$$
\begin{cases} \zeta_1v_m^+u_1^+-f^+\overline{v_1^+}u_1^+\in \mathcal{P}^+,\\
              \zeta_2v_m^+u_2^+-f^+\overline{v_2^+}u_2^+\in \mathcal{P}^+,\\
              \cdot\hskip+1cm \cdot\hskip+1cm \cdot\\
              \zeta_{m-1}v_m^+u_{m-1}^+-f^+\overline{v_{m-1}^+}u_{m-1}^+\in \mathcal{P}^+,\\
              \zeta_1u_1^+v_m^++\zeta_2u_2^+v_m^++\ldots+\zeta_{m-1}u_{m-1}^+v_m^+
              +f^+\overline{u_m^+}v_m^+\in \mathcal{P}^+.
         \end{cases}
$$
Subtracting the first $m-1$ conditions from the last condition in
the latter system, we get
\begin{equation}
f^+(t)\left(\sum_{i=1}^{m-1}
u^+_i(t)\overline{v_i^+(t)}+\overline{u_m^+(t)}v_m^+(t)\right)\in
\mathcal{P}^+.
\end{equation}
Since the second multiplier in (33) belongs to $\mathcal{P}_N$
(see (31)), (21) and (33) imply that
\begin{equation}
\sum_{i=1}^{m-1}u^+_i(t)\overline{v_i^+(t)}+\overline{u_m^+(t)}v_m^+(t)\in
\mathcal{P}_N^+.
\end{equation}
We can interchange the roles of $u$ and $v$ in the above
discussion to get in a similar manner that
\begin{equation}
\sum_{i=1}^{m-1}v^+_i(t)\overline{u_i^+(t)}+\overline{v_m^+(t)}u_m^+(t)\in
\mathcal{P}_N^+.
\end{equation}
It follows from relations (34) and (35) that the function in (32)
belongs to $\mathcal{P}_N^+\cap \mathcal{P}_N^-$, which implies
(32).
\end{proof}

The proof of Theorem 2 {\bf (a)} proceeds as follows. We search
for a nontrivial polynomial solution
\begin{equation}
\mathbf{x}(t)=\big(x^+_1(t),x^+_2(t),\ldots,x^+_m(t)\big)^T\in
\mathcal{P}_N^+
\end{equation}
of the system (28), where
\begin{equation}
x_i^+(t)=\sum_{n=0}^N a_{in} t^n,\;\;\;i=1,2,\ldots, m,
\end{equation}
and explicitly determine the coefficients $a_{in}$. We will find
such $m$ linearly independent  solutions of (28) (see (51) below).

Equating all the non-positive Fourier coefficients of the
functions on the left-hand side of (28) to zero, except the $0$th
coefficient of the $j$th function which we equate to $1$, we get
the following system of algebraic equations in the block matrix
form which we denote by $\mathbb{S}_j$:
\begin{equation}\mathbb{S}_j:=
    \begin{cases}\Gamma_1\cdot X_m-D\cdot\overline{X_1}={\bf 0}, \\
    \Gamma_2\cdot X_m-D\cdot\overline{X_2}={\bf 0}, \\
    \cdot\;\;\;\;\;   \cdot\;\;\;\;\;    \cdot\\
    \Gamma_j\cdot X_m-D\cdot\overline{X_j}={\bf 1}, \\
    \cdot\;\;\;\;\;   \cdot\;\;\;\;\;    \cdot\\
    \Gamma_{m-1}\cdot X_m-D\cdot\overline{X_{m-1}}={\bf 0}, \\
    \Gamma_1\cdot X_1+\Gamma_2\cdot X_2+\ldots+\Gamma_{m-1}\cdot
    X_{m-1}+D\cdot\overline{X_m}={\bf 0}\;. \end{cases}
\end{equation}
Here the following matrix notation is used:
\begin{equation}
D=\begin{pmatrix}d_0&d_1&d_2&\cdots&d_{N-1}&d_N\\
        0&d_0&d_1&\cdots&d_{N-2}&d_{N-1}\\
        0&0&d_0&\cdots&d_{N-3}&d_{N-2}\\
        \cdot&\cdot&\cdot&\cdots&\cdot&\cdot\\
        0&0&0&\cdots&0&d_0\end{pmatrix},\;\;
\Gamma_i=\begin{pmatrix}\gamma_{i0}&\gamma_{i1}&\gamma_{i2}&\cdots&\gamma_{i,N-1}&\gamma_{iN}\\
        \gamma_{i1}&\gamma_{i2}&\gamma_{i3}&\cdots&\gamma_{iN}&0\\
        \gamma_{i2}&\gamma_{i3}&\gamma_{i4}&\cdots&0&0\\
        \cdot&\cdot&\cdot&\cdots&\cdot&\cdot\\
        \gamma_{iN}&0&0&\cdots&0&0\end{pmatrix},
\end{equation} where
$$
 f^+(z)=\sum_{n=0}^N d_n z^n \;\text{ and }\; \zeta_i(t)=
 \sum_{n=-N}^N\gamma_{in}t^{-n},\;\;i=1,2,\ldots,m-1,
$$
\begin{equation}
{\bf 0}=(0,0,\ldots,0)^T\in \mathbb{C}^{N+1}, \text{ and }{\bf
1}=(1,0,0,\ldots,0)^T\in \mathbb{C}^{N+1}.
\end{equation}
The column vectors
\begin{equation}
X_i=(a_{i0},a_{i1},\ldots,a_{iN})^T,\;\;i=1,2,\ldots,m,
\end{equation}
(see (37)) are unknowns.

\begin{remark}
We recall that if $(X_1,X_2,\ldots,X_m)$ defined by $ (41) $ is a
solution of the system $ (38)$, then the vector function $ (36)$
defined by $ (37)$ is a solution of the system $ (28)$.
\end{remark}

We need to show that the system $\mathbb{S}_j$ (see (38))  has a
solution for each $j=1,2,\ldots,m$.

Since $f^+(0)>0$ (see (21)), $\frac{1}{f^+}$ can be represented as
a power series in the neighborhood of $0$
$$  \frac{1}{f^+(z)}=\sum_{n=0}^\infty b_n z^n,    $$
where $b_0=(f^+(0))^{-1}>0$, and the inverse of the matrix $D$ is
\begin{equation}  D^{-1}=\begin{pmatrix}b_0&b_1&b_2&\cdots&b_{N-1}&b_N\\
        0&b_0&b_1&\cdots&b_{N-2}&b_{N-1}\\
        0&0&b_0&\cdots&b_{N-3}&b_{N-2}\\
        \cdot&\cdot&\cdot&\cdots&\cdot&\cdot\\
        0&0&0&\cdots&0&b_0\end{pmatrix}.    \end{equation}

Determining $X_i$, $i=1,2,\ldots,m-1$, from the first $m-1$
equations of (38),
\begin{equation}
X_i=\overline{D^{-1}}\;\overline{\Gamma_i}\;\overline{X_m}-\delta_{ij}\overline{D^{-1}}\;{\bf
1},\;\;i=1,2,\ldots,m-1,
\end{equation}
 and then substituting them into the last equation of (38), we get
$$
\Gamma_1\,\overline{D^{-1}}\;\overline{\Gamma_1}\;\overline{X_m}+\Gamma_2\,\overline{D^{-1}}\;\overline{\Gamma_2}\;
\overline{X_m}
+\ldots+\Gamma_{m-1}\,\overline{D^{-1}}\;\overline{\Gamma_{m-1}}\;\overline{X_m}+D\;\overline{X_m}=
\Gamma_j\,\overline{D^{-1}}\,{\bf 1}
$$
(it is  assumed that $\Gamma_m=\overline{D}$, i.e. the right-hand
side is equal to ${\bf 1}$ when $j=m$) or, equivalently,
\begin{equation}
(\Theta_1\,\overline{\Theta_1}+\Theta_2\,\overline{\Theta_2}
+\ldots+\Theta_{m-1}\,\overline{\Theta_{m-1}}+I_{N+1})\,\overline{X_m}
=D^{-1}\,\Gamma_j\,\overline{D^{-1}}\,{\bf 1},
\end{equation}
where
\begin{equation}
\Theta_i=D^{-1}\,\Gamma_i\,,\;\;i=1,2,\ldots,m-1.
\end{equation}
For each $j=1,2,\ldots,m$, (44) is a linear algebraic system of
$N+1$ equations with $(N+1)$ unknowns.

The matrices $\Theta_i$, $i=1,2,\ldots,m-1$, are symmetric since
their entries are (see (45), (42), and (39))
\begin{equation}
\Theta_i[k,l]=\Theta_i[l,k]=\begin{cases}0\;\;\;&\text{for}\;\;
            k+l>N,\\
    \displaystyle\sum_{n=0}^{N-(k+l)}b_n\gamma_{i,k+l+n}\;\;&\text{for}\;\;
            k+l\leq N.\end{cases}
\end{equation} Therefore
$\Theta_i\overline{\Theta_i}=\Theta_i\Theta_i^*$,
$i=1,2,\ldots,m-1$, are non-negative definite and the coefficient
matrix of the system (44)
\begin{equation}
\Delta=\Theta_1{\Theta_1}^*+\Theta_2{\Theta_2}^*
+\ldots+\Theta_{m-1}{\Theta_{m-1}^*}+I_{N+1}
\end{equation}
(which is the same for each $j=1,2,\ldots,m$) is positive definite
(with all eigenvalues larger than or equal to 1). Consequently,
$\Delta$ is nonsingular, $\det\Delta\geq1$, and the system (44)
 has a unique solution for each $j$. Furthermore, $\Delta$  has a displacement structure of rank $m$
 (see Appendix B) which reduces the computational burden for  solution of the system (44) from
$O(N^3)$ to $O(mN^2)$ (see [14; App. F]).

Finding the matrix vector $\overline{X_m}$ from (44) and then
determining $X_1,X_2,\ldots,X_{m-1}$ from (43), we get the unique
solution of $\mathbb{S}_j$. To indicate its dependence on $j$, we
denote the solution of $\mathbb{S}_j$ by
$(X_1^j,X_2^j,\ldots,X_{m-1}^j,X_m^j)$,
\begin{equation}
X_i^j:=(a_{i0}^j,a_{i1}^j,\ldots, a_{iN}^j)^T,
\;\;\;i=1,2,\ldots,m,
\end{equation}
so that if we construct a matrix function $V(t)$,
\begin{equation}
V(t)=\begin{pmatrix}v^+_{11}(t)&v^+_{12}(t)&\cdots&v^+_{1,m-1}(t)&v^+_{1m}(t)\\
                 v^+_{21}(t)&v^+_{22}(t)&\cdots&v^+_{2,m-1}(t)&v^+_{2m}(t)\\
           \vdots&\vdots&\vdots&\vdots&\vdots\\
           v^+_{m-1,1}(t)&v^+_{m-1,2}(t)&\cdots&v^+_{m-1,m-1}(t)&v^+_{m-1,m}(t)\\[3mm]
           \overline{v^+_{m1}(t)}&\overline{v^+_{m2}(t)}&\cdots&\overline{v^+_{m,m-1}(t)}&\overline{v^+_{mm}(t)}\\
           \end{pmatrix}
\end{equation}
by letting (see (48))
\begin{equation}
v^+_{ij}(t)=\sum_{n=0}^N a_{in}^j t^n, \;\;\;1\leq i,j\leq m
\end{equation}
\big(note that (49) has the structure required in Theorem 2 {\bf
(a)}; see (15), (23)\big), then its modified columns
$\widetilde{V^1}(t),
\widetilde{V^2}(t),\ldots,\widetilde{V^{m-1}}(t)$, and
$\widetilde{V^m}(t)$,
\begin{equation}
\widetilde{V^j}(t)=(v^+_{1j}(t),v^+_{2j}(t),\ldots,v^+_{m-1,j}(t),v^+_{mj}(t)),\;\;j=1,2,\ldots,m,
\end{equation}
are  solutions of the system (28) (see Remark 1). Hence, because
of the last equation in (28),
 \begin{equation}
F(t) V(t)\in \mathcal{P}^{+}
\end{equation} and, by virtue of Lemma 3,
\begin{equation}
\langle V^i(t),V^j(t)\rangle_m=c_{ij},\;i,j=1,2,\ldots,m,
\end{equation}
for each $t\in\mathbb{T}$. Besides, we have
\begin{equation}
\det V(t)=const,\;\;\;t\in\mathbb{T}.
\end{equation}
Indeed, the inclusion
\begin{equation}
\det V(t)\in \mathcal{P}
\end{equation}
is obvious (see (49) and (50)). The relation (52) implies that
(see (18))
$$
f^+(t)\,\det V(t)=\det F(t)\,\det V(t)\in \mathcal{P}^+.
$$
Thus, it follows from (21) and (55) that
 \begin{equation}
\det V(t)\in \mathcal{P}^+.
\end{equation}
Next we have (see (98) below)
\begin{equation}
\big(F^{-1}\big)^*=\begin{pmatrix}1&0&\cdots&0&-\overline{\zeta_{1}}/\overline{f^+}\\
          0&1&\cdots&0&-\overline{\zeta_{2}}/\overline{f^+}\\
           0&0&\cdots&0&-\overline{\zeta_{3}}/\overline{f^+}\\
           \cdot&\cdot&\cdots&\cdot&\cdot\\
           0&0&\cdots&1&-\overline{\zeta_{m-1}}/\overline{f^+}\\
           0&0&\cdots&0&1/\overline{f^+}
           \end{pmatrix}
\end{equation}
and
\begin{equation}
\det \big(F^{-1}\big)^*(t)=\big(\overline{f^+(t)}\big)^{-1}.
\end{equation}
Since the column vectors (51) are solutions of the system (28), we
have

\begin{equation}
\phi^+_{ij}(t)=\zeta_i(t)v_{mj}^+(t)-f^+(t)\overline{v_{ij}^+(t)}\in
\mathcal{P}^+,\;\;1\leq j\leq m,\,1\leq i< m.
\end{equation}
Direct computations give (see (57), (49), and (59))
$$
\big(F^{-1}\big)^*(t)V(t)=\big(\overline{f^+(t)}\big)^{-1}
\begin{pmatrix}
           -\overline{\phi^+_{11}(t)}&-\overline{\phi^+_{12}(t)}&\cdots&-\overline{\phi^+_{1m}(t)}\\[3mm]
                 -\overline{\phi^+_{21}(t)}&-\overline{\phi^+_{22}(t)}&\cdots&-\overline{\phi^+_{2m}(t)}\\[3mm]
           \vdots&\vdots&\vdots&\vdots\\[3mm]
           -\overline{\phi^+_{m-1,1}(t)}&-\overline{\phi^+_{m-1,2}(t)}&\cdots&-\overline{\phi^+_{m-1,m}(t)}\\[3mm]
           \overline{v^+_{m1}(t)}&\overline{v^+_{m2}(t)}&\cdots&\overline{v^+_{mm}(t)}\\
           \end{pmatrix}.
$$
Thus, there exists a matrix function $\Phi^+(t)\in \mathcal{P}^+$
(hence
\begin{equation}
\det\Phi^+(z)\in \mathcal{P}^+\;)
\end{equation}
such that
$$
\big(F^{-1}\big)^*(t)V(t)=\big(\overline{f^+(t)}\big)^{-1}\cdot\overline{\Phi^+(t)}.
$$
Consequently (see (58)),
$$
\big(\overline{f^+(t)}\big)^{-1}\det
V(t)=\big(\overline{f^+(t)}\big)^{-m}\det \overline{\Phi^+(t)}
$$
so that
$$
\big({f^+(t)}\big)^{m-1}\det\overline {V(t)}=\det {\Phi^+(t)}.
$$
Thus, it follows from (21), (55), and (60) that
\begin{equation}
\det \overline{V(t)}\in \mathcal{P}^+\,.
\end{equation}
The relations (56) and (61) imply $\det V(t)\in \mathcal{P}^+\cap
\mathcal{P}^-$ yielding (54).

The matrix function $V(t)$ is not yet unitary, but it can be
easily made  such by multiplying from the right by a constant
matrix. Namely, the matrix $C=(c_{ij})_{i,j=1,2,\ldots,m}$ defined
by (53),
\begin{equation}
C=\big(V^*(t)V(t)\big)^T,\;t\in\mathbb{T},
\end{equation}
is nonsingular. Indeed, if $C$ were singular and ${\bf 0}\not={\bf
w}=(w_1,w_2,\ldots,w_m)\in\mathbb{C}^m$ were such that ${\bf
w}C=0$, then
$$
\big\|\sum_{j=1}^m w_jV^j(t)\big\|^2_{\mathbb{C}^m}={\bf w}C{\bf
w}^*=0
$$
for each $t\in\mathbb{T}$, i.e. the vector functions $V^1(t),
V^2(t), \ldots,V^m(t)$ would be linearly dependent. But this is
impossible since the linear functional ${\bf L}:L^+_\infty\times
L^+_\infty\times\ldots\times L^+_\infty \to\mathbb{C}^m$ which
maps $(x^+_1(t),x^+_2(t), \ldots,x^+_m(t))$ into the $0$th Fourier
coefficients of the functions standing on the left-hand side of
the system (28), i.e. into
$\big(c_0\{\zeta_1(t)x^+_m(t)-f^+(t)\overline{x^+_1(t)}\}$,
$\ldots$, $c_0
\{\zeta_{m-1}(t)x^+_m(t)-f^+(t)\overline{x^+_{m-1}(t)}\}$,
$c_0\{\zeta_1(t)x^+_1(t)+\zeta_2(t)x^+_2(t)+\ldots+f^+(t)\overline{x^+_m(t)}\}\big)
$, transforms $m$ vector functions $V^1(t)$, $V^2(t)$, $\ldots$,
$V^m(t)$  into linearly independent standard bases of
$\mathbb{C}^m$, namely, ${\bf L}\big(V^j(t)\big)
=(\delta_{j1},\delta_{j2},\ldots,\delta_{jm})$, $j=1,2,\ldots,m$,
because of (38). Consequently, $V(1)$ is also nonsingular  since
 $ {C}^T=V^*(1)V(1)$ (see (62)).
 Let
\begin{equation}
U(t)=V(t)\big(V(1)\big)^{-1}.
\end{equation}
Then $U(t)$ is unitary since (see (63), (62))
$$
U^*(t)U(t)=\big((V(1))^{-1}\big)^*V^*(t)V(t)\big(V(1)\big)^{-1}=
\big((V(1))^{-1}\big)^*V^*(1)V(1)\big(V(1)\big)^{-1}=I_m.
$$
Since the matrix $\big(V(1)\big)^{-1}$ is constant,
$U(t)\in\mathcal{U}_m(\mathbb{T})$ has the same structure (15),
(23) as $V(t)$, and
  \begin{equation}
F(t) U(t)\in \mathcal{P}^+
\end{equation}
 holds as well (see (52) and (63)). Moreover,
$\det U(t)=const$, $t\in \mathbb{T}$ (see (54) and (63)), which
implies that $ \det U(t)=1 $ as we have $U(1)=I_m$ (see (63)).
Consequently,
  \begin{equation}
U(t)\in \mathcal{S} \mathcal{U}_m(\mathbb{T}).
\end{equation}

Let now
\begin{equation}
U_{F}(t)=U(t)\cdot\big(FU(0)\big)^{-1}\sqrt{FU(0)(FU(0))^*}.
\end{equation}
The multiplier of $U(t)$ in (66) is a (constant) unitary matrix,
so that $U_{F}(t)\in  \mathcal{U}_m(\mathbb{T})$, it has the
structure (15), (23) (since $U(t)$ has this structure), the
inclusion (24) holds (see (64), (66)), and (25) is valid too (see
(66)). The relation $ \det U_{F}(t)=1$ holds since  $\det
U_{F}(t)=c$ where $|c|=1$ (see (65), (66)), while $c>0$ since we
know that $0<\det \big(FU_{F}\big)(0)=c\cdot\det F(0)=cf^+(0)$
\big(see (25), Corollary 1, and (18)\big) and $f^+(0)>0$ (see
(21)). Consequently, the matrix function $U_{F}(t)\in \mathcal{S}
\mathcal{U}_m(\mathbb{T})$ defined by (66) satisfies the
requirements of Theorem 2 ($\mathbf{a}$) and it has been
constructed explicitly. The proof of the part ($\mathbf{a}$)  is
finished.

\begin{remark}
Note that, as in the case of  $U(t)$, the modified column vectors
of $U_F(t)$ are solutions of the system $(28)$ since this property
of a matrix function is preserved when we multiply it by a
constant matrix from the right.
\end{remark}

\section{Description of the Method}
A brief outline of the method is the following: $S(t)$ is
approximated by $M_N(t)M_N^*(t)$, where $M_N(t)$ is a lower
triangular matrix function with analytic entries on the diagonal
and whose entries below the diagonal have only finite number of
nonzero Fourier coefficients with negative indices, the last
product is represented as $M_N^+(t)(M_N^+)^*(t)$, where an
analytic matrix function $M_N^+(t)$ is  constructed explicitly,
and its convergence to $S^+(t)$ is proved.

Given a matrix spectral density (1), first  the lower-upper
triangular factorization of $S(t)$ is performed,
\begin{equation}
S(t)=M(t)(M(t))^*,
\end{equation}
where
\begin{equation}
M(t)=\begin{pmatrix}f^+_1(t)&0&\cdots&0&0\\
        \xi_{21}(t)&f^+_2(t)&\cdots&0&0\\
        \vdots&\vdots&\vdots&\vdots&\vdots\\
        \xi_{r-1,1}(t)&\xi_{r-1,2}(t)&\cdots&f^+_{r-1}(t)&0\\
        \xi_{r1}(t)&\xi_{r2}(t)&\cdots&\xi_{r,r-1}(t)&f^+_r(t)
        \end{pmatrix}.
\end{equation}
The functions $f_m^+(t)$, $m=1,2,\ldots,r$, on the diagonal are
taken the canonical spectral factors of the positive functions
$\det S_m(t)/\det S_{m-1}(t)$, where $S_0(t)=1$ and
$S_m(t)=\big(S(t)\big)_{m\times m}$, the upper-left $m\times m$
submatrix of $S(t)$. Namely,
\begin{equation}
 f_m^+(z)=\frac{(\det S_m)^+(z)}{(\det S_{m-1})^+(z)}\,,
\end{equation}
where (see (4))
$$
 (\det S_m)^+(z)=\exp\left(\frac 1{4\pi}
\int\nolimits_0^{2\pi}\frac{e^{i\theta}+z}{e^{i\theta}-z}\log \det
S_m(e^{i\theta})\,d\theta\right).
$$
We have $\log \det S_m(t)\in L_1(\mathbb{T})$, $m=0,1,\ldots,r$,
by virtue of (2) (see, e.g., [4; Sect. 5]), so that the functions
$(\det S_m)^+(z)$, and consequently $f_m^+(z)$, $m=1,2,\ldots,r$,
are well defined in (69). The entries $\xi_{ij}$, $2\leq i\leq r$,
$1\leq j< i$, can be found in a standard algebraic way from the
relation (67).

Note that (67) implies  $|f_1^+|^2=s_{11}\in L_1$ and
$\sum_{j=1}^{i-1}|\xi_{ij}|^2+|f_i^+|^2=s_{ii}\in L_1$,
$i=2,3,\ldots, r$. Thus $ M(t)\in L_2(\mathbb{T})$ (and hence $
M(t)U(t)\in L_2(\mathbb{T})$ for any $U(t)\in
\mathcal{U}_r(\mathbb{T})$). Furthermore, $ f^+_m\in
\mathcal{O}_2^O$, $m=1,2,\ldots,r$, which implies that
\begin{equation}
\det M(t)=f^+_1(t)f^+_2(t)\ldots f^+_r(t)\in
\mathcal{O}^O_{2/r}\,.
\end{equation}

We search for ${U}(t)\in \mathcal{S}\mathcal{U}_r(\mathbb{T})$
such that $M(t)U(t)$ is a spectral factor of $S(t)$ and continue
the description of our method in terms of Propositions 1 and 2,
which can be easily proved using Theorems 1 and 2, respectively.

\begin{proposition}
A spectral factor of $S(t)$ can be represented as
\begin{equation}
S^+(t)=M(t)\mathbf{U}_2(t)\mathbf{U}_3(t)\ldots \mathbf{U}_r(t),
\end{equation}
where $\mathbf{U}_m(t)\in \mathcal{S} \mathcal{U}_r(\mathbb{T})$
 has the block matrix form
\begin{equation}
\mathbf{U}_m(t)=\begin{pmatrix}U_{F_m}(t)&0\\0&I_{r-m}\end{pmatrix},\;
m=2,3,\ldots r-1,\;\mathbf{U}_r(t)=U_{F_r}(t),
\end{equation}
$F_m$ in $ (72)$ is the matrix function of the form $(13)$ whose
last row coincides with the last row of
$\big(M_{m-1}(t)\big)_{m\times m}$,
\begin{equation}
M_1(t):=M(t),\;\;M_m(t):=M(t)\mathbf{U}_2(t)\mathbf{U}_3(t)\ldots
\mathbf{U}_m(t)=M_{m-1}(t) \mathbf{U}_m(t),
\end{equation}
and $U_{F_m}(t)\in \mathcal{S} \mathcal{U}_m(\mathbb{T})$ is the
corresponding matrix function  determined according to Theorem
$1$, $m=2,3,\ldots r$.
\end{proposition}

\begin{proof}
Obviously, the product of two matrix functions from $ \mathcal{S}
\mathcal{U}_r(\mathbb{T})$ is in the same class. Thus, by virtue
of Lemma 2 (see (67), (70)), it suffices to show that
\begin{equation}
S^+(t)=M_r(t)=M(t)\mathbf{U}_2(t)\mathbf{U}_3(t)\ldots
\mathbf{U}_r(t)\in L_2^+(\mathbb{T}).
\end{equation}

It follows from the structures of the matrices in (72) and (73)
that
\begin{equation}
\big(M_m(t)\big)_{m\times m}=\big(M_{m-1}(t)\big)_{m\times
m}U_{F_m}(t),
\end{equation}
while the last $r-m$ columns of $M(t)$ remains unaltered in
$M_m(t)$.

We show by induction that
\begin{equation}
\big(M_m(t)\big)_{m\times m}\in L_2^+(\mathbb{T}),\;\;
m=1,2,\ldots,r.
\end{equation}
 Indeed, clearly (76) is correct for $m=1$. Assume
now that (76) holds when $m$ is replaced by $m-1$ in it, i.e.
\begin{equation}
\big(M_{m-1}(t)\big)_{(m-1)\times (m-1)}\in L_2^+(\mathbb{T})\,.
\end{equation}
Then $\big(M_{m-1}(t)\big)_{m\times m}\in L_2(\mathbb{T})$ has the
form
\begin{equation}
\big(M_{m-1}(t)\big)_{m\times m}=\begin{pmatrix}\mu^{+}_{11}(t)&\mu^{+}_{12}(t)&\cdots&\mu^{+}_{1,m-1}(t)&0\\
                 \mu^{+}_{21}(t)&\mu^{+}_{22}(t)&\cdots&\mu^{+}_{2,m-1}(t)&0\\
           \vdots&\vdots&\vdots&\vdots&\vdots\\
           \mu^{+}_{m-1,1}(t)&\mu^{+}_{m-1,2}(t)&\cdots&\mu^{+}_{m-1,m-1}(t)&0\\
           \zeta_{1}(t)&\zeta_{2}(t)&\cdots&\zeta_{m-1}(t)&f_m^{+}(t)
           \end{pmatrix},
\end{equation}
where $\mu^{+}_{ij}(t)\in L_2^{+}({\mathbb T})$,
$i,j=1,2,\ldots,m-1$, by  (77), $\zeta_j(t)\in L_2({\mathbb T})$,
$j=1,2,\ldots,m-1$,  $f_m^{+}\in \mathcal{O}_2^0$ is defined by
(69) (see (68)), and
$\big(\zeta_{1}(t),\zeta_{2}(t),\cdots\zeta_{m-1}(t),f_m^{+}(t)\big)$
is the last row of $F_m(t)$ according to its definition in
Proposition 1 (note that Theorem 1 can be applied to the matrix
function $F_m(t)$). The direct computation shows that (78) is
equal to (see (13))
$$
\begin{pmatrix}\mu^{+}_{11}(t)&\cdots&\mu^{+}_{1,m-1}(t)&0\\
                 \mu^{+}_{21}(t)&\cdots&\mu^{+}_{2,m-1}(t)&0\\
           \vdots&\vdots&\vdots&\vdots\\
           \mu^{+}_{m-1,1}(t)&\cdots&\mu^{+}_{m-1,m-1}(t)&0\\
          0&\cdots&0&1
           \end{pmatrix}F_m(t)=\begin{pmatrix}\big(M_{m-1}(t)\big)_{(m-1)\times (m-1)}&0\\0&1
           \end{pmatrix}F_m(t)
$$
and we get that (see (75))
\begin{equation}
 \big(M_m(t)\big)_{m\times m}=\big(M_{m-1}(t)\big)_{m\times
 m}U_{F_m}(t)=\begin{pmatrix}\big(M_{m-1}(t)\big)_{(m-1)\times (m-1)}&0\\0&1
           \end{pmatrix}F_m(t)U_{F_m}(t)
\end{equation}
belongs to  $L_2^+(\mathbb{T})$ (see (77) and (17)). Thus
 (76) is valid and taking $m=r$ in (76) we get (74).

Proposition 1 is proved.
\end{proof}

\begin{remark} One can see from the above proof that $\big(M_m(t)\big)_{m\times m}$ is a
spectral factor of $S_m(t)$,
$$
 S_m^+(t)=\big(M_m(t)\big)_{m\times m},\;\;\;m=1,2,\ldots,r.
$$
Thus  the representation $ (71)$ realizes the step-by-step
factorization of the upper-left submatrices of $S(t)$.
\end{remark}
\begin{remark}
There is an alternative way of representation $ (71)$ which avoids
preliminary computation of entries $\xi_{ij}(t)$ in $ (68)$.
Namely, we can compute only $f_m^+(t)$, $m=1,2,\ldots,r$, in
 $ (68)$ according to $ (69)$ $($note that  $\big(M_1(t)\big)_{1\times
1}=f_1^+(t)$ $)$ and determine $\big(M_m(t)\big)_{m\times m}$
recurrently from $\big(M_{m-1}(t)\big)_{(m-1)\times (m-1)}$ by the
formula $ (79)$. The entries
$\zeta_{1}(t),\zeta_{2}(t),\cdots\zeta_{m-1}(t)$ of $F_m(t)$ and
of $ (78)$ can be determined from the equation
$$
\big(M_{m-1}(t)\big)_{(m-1)\times (m-1)}
(\zeta_1,\zeta_2\ldots,\zeta_{m-1})^*=(s_{1m},s_{2m},\ldots,s_{m-1,m})^T
$$
which follows from  $\big(M_{m-1}(t)\big)_{m\times
m}\big(M_{m-1}(t)\big)^*_{m\times m}=S_m(t)$. In this way, we can
obtain each $\big(M_m(t)\big)_{m\times m}$, $m=1,2,\ldots,r$, and
respectively $S^+(t)=\big(M_r(t)\big)_{r\times r}=M_r(t)$.
\end{remark}

Relying on Proposition 1, we recurrently approximate $S^+(t)$ as
follows. Let $N_2$,$N_3$,$\ldots$, $N_r$  be large positive
integers, and let
\begin{equation}
\hat{S}^+(t)=\hat{S}^+[N_2,N_3,\ldots,N_r]:=
{M}(t)\hat{\mathbf{U}}_2(t)\hat{\mathbf{U}}_3(t)\ldots
\hat{\mathbf{U}}_r(t),
\end{equation}
where $\hat{\mathbf{U}}_m(t)\in \mathcal{S}
\mathcal{U}_r(\mathbb{T})$
 has the block matrix form
\begin{equation}
\hat{\mathbf{U}}_m(t)=\begin{pmatrix}U_{\hat{F}_m^{(N_m)}}(t)&0\\[1mm]0&I_{r-m}\end{pmatrix},
\; m=2,3,\ldots
r-1,\;\hat{\mathbf{U}}_r(t)=U_{\hat{F}^{(N_r)}_r}(t),
\end{equation}
$\hat{F}_m\in L_2(\mathbb{T})$  is the matrix function of the form
$ (13)$ whose last row coincides with the last row of
$\big(\hat{M}_{m-1}(t)\big)_{m\times m}$,
\begin{equation}
\hat{M}_1(t):=M(t),\;\;\hat{M}_m(t):=M(t)\hat{\mathbf{U}}_2(t)\hat{\mathbf{U}}_3(t)\ldots
\hat{\mathbf{U}}_m(t)=\hat{M}_{m-1}(t)\hat{\mathbf{U}}_m(t),
\end{equation}
$\hat{F}^{(N_m)}_m(t)$ in (81) is $P_{N_m}\hat{F}_m(t)$ (see the
definition of the projection operator $P_N$ in Section 3), and
$U_{\hat{F}^{(N_m)}_m}(t)\in \mathcal{S}
\mathcal{U}_m(\mathbb{T})$ is the corresponding matrix function
determined according to Theorem $2$ {\bf (a)}, $m=2,3,\ldots r$.
We emphasize that as $S(t)$ is given and the positive integers
$N_2,N_3,\ldots,N_r$ are fixed, each
$\hat{M}_m(t)=\hat{M}_m[N_2,N_3,\ldots,N_m]$, $m=2,3,\ldots,r$,
can be explicitly constructed according to the proof of Theorem 2
($\mathbf{a}$). The following proposition shows that
\begin{equation}
\hat{S}^+(t)=\hat{S}^+[N_2,N_3,\ldots,N_r]=\hat{M}_r(t)
\end{equation}
 (see (80) and (82)) approximates a spectral factor of $S(t)$.

\begin{proposition}
$ \|{S}^+(t)-\hat{S}^+(t)\|_{L_2}\to 0\;\;\text{ as }
N_2,N_3,\ldots,,N_{r}\to\infty. $
\end{proposition}

\begin{proof}
We  prove by induction that
\begin{equation}
\|M_{m}(t)-\hat{M}_{m}(t)\|_{L_2}\to 0\text{ as }
N_2,N_3,\ldots,N_{m}\to\infty,\;m=2,3,\ldots,r.
\end{equation}
Indeed, assume that
\begin{equation}
\|M_{m-1}(t)-\hat{M}_{m-1}(t)\|_{L_2}\to 0\text{ as }
N_2,N_3,\ldots,N_{m-1}\to\infty
\end{equation}
holds (note that $M_{1}(t)=\hat{M}_{1}(t)$). Then, by virtue of
the definitions of ${F}_m(t)$ and $\hat{F}_m(t)$,
\begin{equation} \|{F}_m(t)-\hat{F}_m(t)\|_{L_2}\to 0\;\;\text{ as }
N_2,N_3,\ldots,N_{m-1}\to\infty.
\end{equation}
Obviously (see (20)),
\begin{equation}
\|\hat{F}_m(t)-\hat{F}_m^{(N)}(t)\|_{L_2}\to
0\;\;\text{as}\;\;N\to \infty.
\end{equation}
It follows from (86) and (87) that
$$
\|{F_m}(t)-\hat{F}_m^{(N_m)}(t)\|_{L_2}\to
0\;\;\text{as}\;\;N_2,N_3, \ldots,N_{m}\to\infty.
$$
Thus, by virtue of Theorem 2 $\mathbf{(b)}$,
$U_{\hat{F}_m^{(N_m)}}(t)\rightrightarrows U_{F_m}(t)$ and hence
(see (72) and (81)) $\hat{\mathbf{U}}_m(t)\rightrightarrows
\mathbf{U}_m(t)$ as $N_2,N_3$,$\ldots$,$N_{m}$ $\to$ $\infty$.
Consequently (see (73), (82), (85), and Statement 1 in Sect 3),
$$
 \|{M}_m(t)-\hat{M}_m(t)\|_{L_2}= \|{M}_{m-1}(t)
 {\mathbf{U}}_m(t)-\hat{M}_{m-1}(t)
 \hat{\mathbf{U}}_m(t)\|_{L_2}\to 0
$$
as $N_2,N_3,\ldots,N_m\to \infty$ and (84) holds.

If we substitute $m=r$ into (84), we get the proposition (see (74)
and (83)).
\end{proof}

\begin{remark}
The rate of convergence in Proposition 2 is estimated under minor
restrictions on $S(t)$, which is the subject of a forthcoming
paper.
\end{remark}
\begin{remark}
In actual computations of $\hat{S}^+(t)$ according to $ (80)$, we
cannot take $M(t)$ exactly since  it requires scalar spectral
factorizations. As it was mentioned above, our method does not
contain any improvement in approximate computation of $M(t)$. We
can assume that it can be constructed
$\hat{M}_1(t)=\hat{M}_1[N_1](t)$ in $ (82)$ such that
$\|{M}(t)-\hat{M}_1(t)\|_{L_2}\to 0$ as $N_1\to \infty$, and the
rest of the proof of Proposition $2$ goes through without any
change.
\end{remark}

If we wish to construct an approximation to the canonical spectral
factor ${S}_c^+$, then we take (see Lemma 1)
 $
\hat{S}_c^+(z)=\hat{S}^+(z)\big(\hat{S}^+(0)\big)^{-1}\sqrt{\hat{S}^+(0)(\hat{S}^+(0))^*}\,.
 $

\section{Numerical Simulations}

The computer code for the factorization of polynomial matrix
functions by our method was written in  MatLab in order to test
the algorithm numerically and compare it with other existing
software implementations available in the MatLab toolbox ``Polyx".
The results of our numerical simulations are presented in this
section.

Two different commands, ${\rm spf}(\cdot)$ and ${\rm
spf}(\cdot,syl)$ are available in Polyx to perform polynomial
matrix spectral factorization for a discrete time variable $z$.
(As it is explained in the software manual these factorizations
are based on the Newton-Raphson iteration and on the Sylvester's
method, respectively.) We have supplied the three programs with
the same data and compared their performances. The computer with
characteristics Intel(R) Core(TM) Quad CPU,   Q6600 \@2.40GHz,
2.40 GHz, RAM 2.00Gb was used for these simulations.

In the first place we took a test matrix  whose spectral
factorization was known beforehand,
$$
\begin{pmatrix}2z^{-1}+6+2z&7z^{-1}+22+11z\\
11z^{-1}+22+7z&38z^{-1}+84+38z\end{pmatrix}=\begin{pmatrix}2+z^{-1}&1\\
7+5z^{-1}&3+z^{-1}\end{pmatrix}
\begin{pmatrix}2+z&7+5z\\
1&3+z\end{pmatrix}
$$
(the matrix is very simple, but its determinant, $-z^{-2}+2-z^2$,
has two double zeros on the boundary, which usually causes
difficulties in many methods). So  the correct answer for the
(right) spectral factor (with the uniqueness restriction for the
coefficient matrix of the highest degree of $z$ to be upper
triangular with positive entries on the diagonal, as it is in
Polyx) is {\small
$$
\begin{pmatrix}2.000\ldots&7.000\ldots\\
1.000\ldots&3.000\ldots\end{pmatrix}+\begin{pmatrix}1.000\ldots&5.000\ldots\\
0&1.000\ldots\end{pmatrix}t\,.
$$}
The resulting coefficient matrices obtained by ${\rm spf}(\cdot)$
and ${\rm spf}(\cdot,syl)$ were the same {\small
$$
\begin{pmatrix}1.998657938438840&   6.995302784535943\\
   1.002013092341746 &  3.007856444481516\end{pmatrix}, \begin{pmatrix}1.000671030780579 &
      5.001537986446627\\ 0 &  1.001621242570816\end{pmatrix}
$$}and the time elapsed varied within 0.22-0.24 sec. Below we present
the results of computation by the program based on our algorithm
which shows the advantage of the proposed method. In the process
of the calculations three different pairs of tuning parameters
were used: $\epsilon$, the accuracy level of scalar spectral
factorizations of $S_1(t)$ and $\det S_2(t)$ in (69), and $N=N_2$,
a  positive integer in (83). Accuracy improvements are evident as
proved theoretically in Section 6:

\noindent $\epsilon=0.0001$; $N=20$; time elapsed: 0.04  sec.
{\small
$$
\begin{pmatrix}1.999540036680776 &   6.998506647352555\\
   1.000742648478974 &      3.002561886914702\end{pmatrix}, \begin{pmatrix}
   1.000213381614498&
   5.000347044583584\\ 0 &  1.000664442251186\end{pmatrix}
$$}
\noindent $\epsilon=0.000001$; $N=30$;    time elapsed: 0.14 sec
{\small
$$
\begin{pmatrix}1.999999670218186&   6.999998845763626\\
   1.000000494672526 & 3.000001951212981\end{pmatrix},
   \begin{pmatrix}  1.000000439718603
   &5.000000357258912
  \\ 0 &   1.000000164890947\end{pmatrix}
$$}
\noindent $\epsilon=0.00000001$; $N=40$;  time elapsed: 0.31 sec
{\small
$$
\begin{pmatrix}1.999999957974826&   6.999999852911864 \\
  1.000000063037774 &  3.000000246615687\end{pmatrix}, \begin{pmatrix} 1.000000021012597&
  5.000000047560575 \\ 0 &  1.000000051966992\end{pmatrix}
$$}

When data were selected at random and exact results were unknown,
the mean of absolute values of polynomial coefficients of the
error matrix $\big(\hat{S}^+\big)^*\hat{S}^+-S$ was taken in the
capacity of an  accuracy estimator (in general, the closeness of
$\big(\hat{S}^+\big)^*\hat{S}^+$ to $S$ does  not imply that
$\hat{S}^+$ is close to $S^+$, see [13], [1], but this is the case
for polynomial matrix functions). In the table below this mean is
denoted by $\varepsilon$. Calculation time values are shown, and
the matrix sizes are given; say $4\times 10$ indicates that a
$4\times 4$ test matrix was selected with (Laurent) polynomial
entries of degree 10 (with coefficients from -10 to 10). The
results of calculations by ${\rm spf}(\cdot)$ and ${\rm
spf}(\cdot,syl)$ were almost identical. In the case of our
algorithm, we varied the tuning parameters of the program ($N_2,
N_3,\ldots,N_r$ in (83)) so as to obtain a slightly higher
accuracy than by ${\rm spf}(\cdot)$ and ${\rm spf}(\cdot,syl)$,
while the advantage in time was noticeable.

{\fontsize{9}{9pt}\selectfont
\begin{center}
\begin{tabular}{|c|c|c|c|c|c|c|c|c|c|c|c|c|c|}
\hline
&&&&&&&&&&&\\[-2mm]
&matr&time&accur.&matr.&time&accur.&matr&time&accur.&matr.&time&accur.\\
&size&$sec$&$\varepsilon$&size&$sec$&$\varepsilon$&size&$sec$&$\varepsilon$&size&$sec$&$\varepsilon$\\[2mm]
\hline &&&&&&&&&&&\\[-2mm]
$spf(\cdot)$&4x10&0.67&$10^{-10}$&6x15&5.8&$10^{-6}$&10x20&218&$10^{-6}$&15x20&1949&$10^{-7}$
\\\hline &&&&&&&&&&&\\[-2mm]
$spf(\cdot,syl)$&--&0.56&$10^{-10}$&--&4.7&$10^{-6}$&--&214&$10^{-6}$&--&1952&$10^{-7}$\\
\hline &&&&&&&&&&&\\[-2mm]
New Alg.&--&0.46&$10^{-12}$&--&3.4&$10^{-7}$&--&65&$10^{-8}$&--&216&$10^{-8}$\\
\hline
\end{tabular}
\end{center}}

\smallskip

We express our gratitude to PhD student Vakhtang Rodonaia for
working out the software for testing our algorithm and collecting
the numerical data.

\section{Conclusion}

A new algorithm of matrix spectral factorization is developed,
which factorizes any matrix spectral density that admits spectral
factorization. The advantage of the algorithm is illustrated by
the examples of numerical simulations.

\section{Appendices}
{\bf A. Convergence properties. } In this appendix we continue the
proof of Theorem 2 started in Sect. 5 and prove the second part
$\mathbf{(b)}$, which deals with convergence properties of the
algorithm. This proof is similar to the one given in [12] for the
two-dimensional case.

Observe first that:

(i) {\em if $\{U_{F^{(N)}}(t)\}_{N\in\mathbb{N}_0}$,
$\mathbb{N}_0\subset\mathbb{N}$, is any convergent almost
everywhere subsequence of $U_{F^{(N)}}(t)$, i.e. if
$$
U_{F^{(N)}}(t)\to U(t) \;\;\text{a.e. as}\;\;\mathbb{N}_0\ni N\to
\infty,
$$
then}
\begin{equation}
F_c^+(t)=F(t)U(t).
\end{equation}
Indeed, passing to the limit in the relations (22), (24), $\det
\big(F^{(N)}U_{F^{(N)}}\big)(z)=f_{(N)}^+(z)$ (see Corollary 1 and
 (18)), and (25), we get
$$
U(t)\in \mathcal{S} \mathcal{U}_m(\mathbb{T}),\;  F(t)U(t)\in
L_2^+(\mathbb{T}), \; \big(FU\big)(z)=f^+(z)\in \mathcal{O}_2^0
\text{ (see (14)), and } FU(0)>0,
$$
which implies (88) (see Lemma 2).

Now it will be shown that

(ii) {\em from each subsequence
$\{U_{F^{(N)}}(t)\}_{N\in\mathbb{N}_1\subset\mathbb{N}}$ we can
extract an a.e. convergent subsequence
$\{U_{F^{(N)}}(t)\}_{N\in\mathbb{N}_0\subset\mathbb{N}_1}$}.

This will finish the proof of the relation (26) by virtue of the
uniqueness of the canonical spectral factor and  the property (i).

We say that a sequence of functions $f_n\in L_2$, $n=1,2,\ldots$
belongs to $\mathcal{K}$, $\{f_n\}_{n\in \mathbb{N}}\in
\mathcal{K}$, if one can extract a convergent in $L_2$ subsequence
from $f_n$. Recall that an operator $K:L_\infty^\pm\to L_2^\pm$ is
called compact if $\{K(h_n)\}_{n\in \mathbb{N}}\in \mathcal{K}$
for any bounded sequence $\{h_n\}_{n\in \mathbb{N}}$, $|h_n|<c$,
$n=1,2,\ldots$ (see [15; \S 4.6]).

To prove the property (ii), observe that Hankel's operators
$$
H_\zeta:L_\infty^+\to L_2^-,\; \zeta\in
L_2,\;\;\;\text{and}\;\;\;H_f^*:L_\infty^-\to L_2^+,\; f\in L_2,
$$
defined by
\begin{equation}
H_\zeta(u^+)=P^-(\zeta u^+),\;\;\;u^+\in L_\infty^+,
\end{equation}
and
\begin{equation}
H^*_f(u^-)=P^+(f u^-),\;\;\;u^-\in L_\infty^-,
\end{equation}
are compact operators as  limits of finite-dimensional operators
(see, e.g., [15; Th. 4.6.1]).

Fix arbitrary $j\leq m$, and let
$({u_1^{+(N)}},{u_2^{+(N)}},\ldots,u^{+(N)}_{m-1},\overline{{u_m^{+(N)}}})^T$
be the $j$th column of $U_{F^{(N)}}$. Since the modified columns
of $U_{F^{(N)}}(t)$ are solutions of the system (28) (see Remark 2
in Sect. 5), we have
\begin{equation}
\zeta_i^{(N)}{u^+_m}^{(N)}-f_{(N)}^+\overline{{u_i^{+(N)}}}\in
\mathcal{P}^+,\;\;1\leq i\leq m-1,
\end{equation}
and
\begin{equation}
\zeta_1^{(N)}{u_1^{+(N)}}+\zeta_2^{(N)}{u_2^{+(N)}}+
\ldots+\zeta_{m-1}^{(N)}u^{+(N)}_{m-1}+f_{(N)}^+\overline{{u_m^{+(N)}}}\in
\mathcal{P}^+.
\end{equation}
It follows from the compactness of  the operator (89) and (19)
that
\begin{equation}
\left\{P^-
\big(\zeta_i^{(N)}{u_i^{+(N)}}\big)\right\}_{N\in\mathbb{N}_1}=
\left\{P^- \big((\zeta_i^{(N)}-\zeta_i){u_i^{+(N)}}\big) +P^-\big
(\zeta_i{u_i^{+(N)}}\big)\right\}_{N\in\mathbb{N}_1}\in
\mathcal{K}
\end{equation}
 for each  $i=1,2,\ldots,m-1$, and thus
 $\{P^-(f_{(N)}^+\overline{{u^+_m}^{(N)}})\}_{N\in\mathbb{N}_1}\in
\mathcal{K}$, because of the relation (92). It follows from the
compactness of operator (90) and (19) that
\begin{equation}
\left\{P^+\big(f_{(N)}^+\overline{{u_m^{+(N)}}}\big)\right\}_{N\in\mathbb{N}_1}=
\left\{P^+\big((f_{(N)}^+-f^+)\overline{{u_m^{+(N)}}}\big)+
P^+\big(f^+\overline{{u_m^{+(N)}}}\big)\right\}_{N\in\mathbb{N}_1}\in
\mathcal{K}
\end{equation}
as well. Hence (see (93), (94))
\begin{equation}
\left\{f_{(N)}^+\overline{{u^+_m}^{(N)}}\right\}_{N\in\mathbb{N}_1}=\left\{P^+(f_{(N)}^+\overline{{u^+_m}^{(N)}})+
 P^-(f_{(N)}^+\overline{{u^+_m}^{(N)}})-c_0(f_{(N)}^+\overline{{u^+_m}^{(N)}})\right\}_{N\in\mathbb{N}_1}\in
\mathcal{K}.
\end{equation}
Since $f_{(N)}^+(t)\rightrightarrows f^+(t)$ and
  $f^+(t)\not=0$ for a.a. $t\in\mathbb{T}$ (see (14)),
 it follows from (95) that $\{{u^+_m}^{(N)}\}_{N\in\mathbb{N}_1}$ contains an almost
 everywhere convergent subsequence.

Now we will show that the same is true for
$\{{u_i^{+(N)}}\}_{N\in\mathbb{N}_1}$,  $1\leq i\leq m-1$. Since
$\{\zeta_i^{(N)}{u^+_m}^{(N)}\}_{N\in\mathbb{N}_1}\in \mathcal{K}$
(see (19) and Statement 1) and hence
$\{P^-(\zeta_i^{(N)}{u^+_m}^{(N)})\}_{N\in\mathbb{N}_1}\in
\mathcal{K}$, it follows from (91) that
\begin{equation}
\left\{P^-(f_{(N)}^+\overline{{u_i^{+(N)}}})\right\}_{N\in\mathbb{N}_1}\in
\mathcal{K} \end{equation}
 as well. The compactness of the operator (90) and (19)
imply that
\begin{equation}
\left\{P^+\big(f_{(N)}^+\overline{{u_i^{+(N)}}}\big)\right\}_{N\in\mathbb{N}_1}=
\left\{P^+\big((f_{(N)}^+-f^+)\overline{{u_i^{+(N)}}}\big)+
P^+\big(f^+\overline{{u_i^{+(N)}}}\big)\right\}_{N\in\mathbb{N}_1}\in
\mathcal{K}
\end{equation}
 The relations (96) and (97) imply that
$\{f_{(N)}^+\overline{{u_i^{+(N)}}}\}_{N\in\mathbb{N}_1}\in
\mathcal{K}$ and, consequently, an almost everywhere convergent
subsequence can be extracted from
$\{\overline{{u_i^{+(N)}}}\}_{N\in\mathbb{N}_1}$.

The proof of the property $(ii)$ is completed and thus (26) holds.

\smallskip

The proof of the remaining conditions in the part $\mathbf{(b)}$
 continues as follows. Since the inverse of a matrix
function $F$ of  the form (13) is
\begin{equation}
F^{-1}=\begin{pmatrix}1&0&\cdots&0&0\\
          0&1&\cdots&0&0\\
           0&0&\cdots&0&0\\
           \cdot&\cdot&\cdots&\cdot&\cdot\\
           0&0&\cdots&1&0\\
           -\frac{\zeta_{1}}{f^+}&-\frac{\zeta_{2}}{f^+}&\cdots&-\frac{\zeta_{m-1}}{f^+}&\frac1{f^+}
           \end{pmatrix},
\end{equation}
(19) implies that
\begin{equation}
\big(F^{(N)}\big)^{-1}(t)\rightrightarrows F^{-1}(t).
\end{equation}
Hence
$$
U_{F^{(N)}}(t)=\big(F^{(N)}\big)^{-1}(t)\cdot F^{(N)}(t)\cdot
U_{F^{(N)}}(t)\rightrightarrows F^{-1}(t)F_c^+(t)
$$
(see (99) and (26)) and if we  denote $U_F(t):=F^{-1}(t)F_c^+(t)$,
then (27) holds and the equation in (17) follows  directly. Since
each matrix function in (22) has the structure (15), (23), the
limiting matrix function $U_F(t)\in \mathcal{S}
\mathcal{U}_m(\mathbb{T})$ has the structure (15),
 (16). The uniqueness
of $U_F(t)$ follows from the uniqueness of the canonical spectral
factor and the equation in (17) since $F(t)$ is invertible.

\medskip

{\bf B. Displacement Structure. } In this section we prove that
the matrix $\Delta$ defined by (47) has a {\em displacement
structure of rank $m$ with respect to $Z$}, i.e. (see [14; App.
F.1])
\begin{equation}
R_Z\Delta:=\Delta-Z\Delta Z^*
\end{equation}
has rank $m$, where $Z$ is the upper triangular $(N+1)\times(N+1)$
matrix with ones on the first up-diagonal and zeros elsewhere
(i.e. a Jordan block with eigenvalue $0$). There are several forms
of displacement structure and we have selected a suitable one.

Obviously, $I_{N+1}$ has the displacement structure of rank 1,
namely,
\begin{equation}
R_ZI_{N+1}=I_{N+1}-ZI_{N+1}Z^*={\mathcal E}{\mathcal E}^*,
\end{equation}
where ${\mathcal E}=(0,0,\ldots,0,1)^T \in \mathbb{C}^{N+1}$. We
will show that for each Toeplitz-like matrix
\begin{equation}
\Theta=\begin{pmatrix}\eta_{0}&\eta_{1}&\eta_{2}&\cdots&\eta_{n-1}&\eta_{n}\\
        \eta_{1}&\eta_{2}&\eta_{3}&\cdots&\eta_{n}&0\\
        \eta_{2}&\eta_{3}&\eta_{4}&\cdots&0&0\\
        \cdot&\cdot&\cdot&\cdots&\cdot&\cdot\\
        \eta_{n}&0&0&\cdots&0&0\end{pmatrix},
\end{equation}
the matrix $\Theta\Theta^*$ has the displacement structure of rank
1, namely
\begin{equation} R_Z(\Theta\Theta^*)=\Theta\Theta^*-Z\Theta\Theta^*Z^*=\Lambda\Lambda^*,
\end{equation}
where $\Lambda=(\eta_0,\eta_1,\ldots,\eta_n)^T$. Indeed, it
follows from the definitions of matrices $Z$, $\mathcal{E}$, and
$\Lambda$ and from the structure of $\Theta$ that
\begin{gather}
Z^*\mathcal{E}=\mathbf{0}\;\text{ and }\;\mathcal{E}^TZ=\mathbf{0}^T,\\
Z\Theta=\Theta Z^*\;\text{ and }\;Z\Theta^*=\Theta^* Z^*,
\end{gather}
and
\begin{equation}
\Theta-Z\Theta Z=\Lambda\mathbf{1}^T\;\text{ and }\;
\Theta^*-Z^*\Theta^* Z^*=\mathbf{1}\Lambda^*,
\end{equation}
where $\mathbf{0}$ and  $\mathbf{1}$ are defined by (40). Since
$\mathbf{1}^T\mathbf{1}=1$, it follows from (106) that
$$
(\Theta-Z\Theta Z)(\Theta^*-Z^*\Theta^* Z^*)=\Lambda\Lambda^*.
$$
Hence, taking into account (105),
\begin{gather*}
\Lambda\Lambda^*=\Theta\Theta^*-\Theta Z^*\Theta^*Z^*-Z\Theta
Z\Theta^*+Z\Theta
ZZ^*\Theta^*Z^*=\\
\Theta\Theta^*-Z\Theta\Theta^*Z^*-Z\Theta \Theta^*Z^*+Z\Theta
ZZ^*\Theta^*Z^*
\end{gather*}
and (103) holds since (see (101), (105), and (104))
\begin{gather*}
-Z\Theta \Theta^*Z^*+Z\Theta
ZZ^*\Theta^*Z^*=Z\Theta(ZZ^*-I_{N+1})\Theta^*Z^*\\
=-Z\Theta{\mathcal E}{\mathcal E}^T\Theta^*Z^*=-\Theta
Z^*{\mathcal E}{\mathcal
E}^TZ\Theta^*=-\Theta\,\mathbf{0}\,\mathbf{0}^T\Theta^*=0.
\end{gather*}

Every matrix $\Theta_i$, $i=1,2,\ldots,m-1$, defined by (45) has
the structure (102) by virtue of (46). Thus we can write  (103)
for each $i$,
\begin{equation}
R_Z(\Theta_i\Theta_i^*)=
\Theta_i\Theta_i^*-Z\Theta_i\Theta^*_iZ^*=\Lambda_i\Lambda^*_i\,,\;\;\;i=1,2,\ldots,m-1.
\end{equation}
Since $R_Z$ defined by (100) is linear,
$R_Z(\Delta_1+\Delta_2)=R_Z\Delta_1+R_Z\Delta_2$, it follows from
(47), (107), and (101) that
\begin{gather*}
R_Z\Delta=R_Z\left(\sum_{i=1}^{m-1}\Theta_i\Theta_i^*+I_{N+1}
\right)=\sum_{i=1}^{m-1}R_Z(\Theta_i\Theta_i^*)+R_ZI_{N+1}=
\sum_{i=1}^{m-1}\Lambda_i\Lambda_i^*+\mathcal{E}\mathcal{E}^*=AA^*,
\end{gather*}
where
$A=\left[\Lambda_1,\Lambda_2,\ldots,\Lambda_{m-1},\mathcal{E}\right]$
is the $(N+1)\times m$ matrix (of rank at most $m$) with columns
$\Lambda_1,\Lambda_2,\ldots,\Lambda_{m-1}$ and $\mathcal{E}$.

 \vskip+2cm

 A. Razmadze Mathematical Institute

Georgian National Academy of Sciences

1, M. Aleksidze Str.

Tbilisi 0193

Georgia

E-mail: lephremi@umd.edu

\end{document}